\documentclass[aps,pre,reprint,groupedaddress,showpacs,showkeys]{revtex4-1}
\usepackage{graphicx, color}
\usepackage{amsmath}
\usepackage{amsfonts}
\usepackage[breaklinks=true,colorlinks=true,linkcolor=blue,urlcolor=blue,citecolor=blue]{hyperref}
\usepackage{caption/subcaption}

\newtheorem{theorem}{Theorem}[section]
\newtheorem{lemma}[theorem]{Lemma}

\newtheorem{corollary}[theorem]{Corollary}

\newenvironment{proof}[1][Proof]{\begin{trivlist}
\item[\hskip \labelsep {\bfseries #1}]}{\end{trivlist}}

\newcommand{\qed}{\nobreak \ifvmode \relax \else
      \ifdim\lastskip<1.5em \hskip-\lastskip
      \hskip1.5em plus0em minus0.5em \fi \nobreak
      \vrule height0.75em width0.5em depth0.25em\fi}

\begin{document}

\title{Existence and Stability of Periodic Orbits in $N$-Dimensional Piecewise Linear Continuous Maps}

\author{Arindam Saha}
\email{arindamsaha1507@gmail.com}

\author{Soumitro Banerjee}
\email{soumitro@iiserkol.ac.in}

\affiliation{Department of Physical Sciences \\ Indian Institute of Science Education and Research Kolkata \\ Mohanpur Campus, Nadia--741246, West Bengal, India}



\begin{abstract}
Piecewise smooth maps are known to exhibit a wide range of dynamical
features including numerous types of periodic orbits. Predicting
regions in parameter space where such periodic orbits might occur and
determining their stability is crucial to characterize the dynamics of
the system. However, obtaining the conditions of existence and
stability of these periodic orbits generally use brute force methods
which require successive application of the iterative map on a
starting point. In this article, we propose a faster and more elegant
way of obtaining those conditions without iterating the complete
map. The method revolves around direct computation of higher powers of
matrices without computing the lower ones and is applicable on any
dimension of the phase space. In the later part of the article, we
compare the speed of the proposed method with the other popular
algorithms which shows the effectiveness of the proposed method in
higher dimensions. We also illustrate the use of this method in
computing the regions of existence and stability of a particular class
of periodic orbits in three dimensions.
\end{abstract}

\maketitle

\section{Introduction}

Presence of stable and unstable periodic orbits play a vital role in
determining the dynamical properties of a system. While presence of
stable periodic orbits form the basins of attraction, the unstable
periodic orbits play a crucial role in forming chaotic orbits and the
basin boundaries of the attractors. In fact appearance, disappearance
or change in stability of these periodic orbits are responsible for
most bifurcation phenomena. Hence an important component in
characterising a given system is determining the parameters for which
a particular periodic orbit might exist.

Piecewise smooth maps are useful in describing systems whose evolution
is given by different smooth functions in different regions of
space. Examples of piecewise smooth systems include switching
electrical circuits
\cite{deane1990instability,kousaka1999bifurcation}, impacting
mechanical systems \cite{nordmark1991non}, walking robots
\cite{thuilot1997bifurcation,garcia1998simplest}, cardiac dynamics
\cite{sun1995alternans} and neural spiking
\cite{borgers2003synchronization}. Apart from their applicability,
piecewise smooth maps have also attracted attention due to their rich
dynamical properties. Even maps which are piecewise linear and have
non-linearities only across a switching manifold, exhibit features
like robust chaos \cite{banerjee1998robust} and existence of
infinitely many co-existing attractors
\cite{simpson2014sequences}. Study of piecewise linear maps attains
even more importance as they can describe the behaviour of the
piecewise smooth systems near the border separating the partitions of
the phase space, in which case the matrices take a specific form. This
article aims to study the conditions of existence and stability of
periodic orbits in such piecewise linear maps in any arbitrary dimension.

Periodic orbits in piecewise linear one, two and three dimensional
maps have been studied
\cite{nusse1992border,roy2008border,avrutin2012occurrence}. Some of
the works have also computed the regions in parameter space where
specific types of orbits exist and are stable
\cite{avrutin2012occurrence,gardini2010border,ganguli2005dangerous,panchuk2015bifurcation}. However
most of the studies used repeated application of the iterative map to
obtain them.  An iterative technique, invented by the Russian
mathematician Leonov
\cite{leonov1959map,leonov1962discontinuous,panchuk2015bifurcation}
has also been used to obtain the existence condition of periodic
orbits
\cite{gardini2010border,avrutin2010self,tramontana2012period}. Moreover,
the methods used there were specific to the type of orbit being
analysed.

In this work, we seek to obtain a generic and faster method of finding
the existence and stability criteria for a general class of periodic
orbits. The technique developed might be extended to any periodic obit
in piecewise linear maps of any dimension. Moreover, as the technique
is essentially an algorithm to obtain higher iterates of the map
without going through the intermediate ones, it can be used to
simplify extraction of other relevant information about the
map. Additionally, the maps we work on are in their normal
forms. Hence the ambit of the results obtained spreads across all
piecewise smooth maps as appropriate coordinate transformation near
the border of non-smoothness yields the normal form of the maps.

The paper is organized as follows. After defining the notations and
conventions used in the article in the Section-2, we derive the
generic conditions for existence and stability of periodic orbits in a
dimension independent form. We then develop a technique which exploits
the form of the map to obtain the conditions of existence and
stability in $N$ dimensions. We showcase the utility of the technique
developed by computing the parameter regions in which certain
classes of stable periodic orbits exist in a three dimensional
piecewise linear map in the next section. Thereafter, we take the
specific case of two dimensional systems, where the simplicity of the
map allows us to obtain stronger analytical results. Finally, we
devote a section to exhibit the computational efficiency of the
technique developed by comparing it with other traditionally used
methods for computing the existence and stability criteria of periodic
orbits in piecewise linear maps.

\section{The Piecewise Linear Map in the Normal Form}

In earlier literature it has been shown that, on proper choice of axes, any $N$ dimensional piecewise smooth continuous map can be linearised near the border to have the following form \cite{nusse1992border}
\begin{equation}
  G_\mu(X)=\left\{
  \begin{array}{lr}
   M_L X + \zeta  & : x_1 \le 0\\
~&~\\
   M_R X + \zeta & : x_1 > 0
  \end{array} \right.
  \label{eq: Map}
\end{equation}
where $M_L, M_R \in \mathbb{R}^{N \times N}$ are matrices in their normal form given by \cite{di2003normal}
\begin{equation}
M_J =
\begin{pmatrix}
- d^{(J)}_1 & 1 & 0 & \cdots  & 0 \\
- d^{(J)}_2 & 0 & 1 & \cdots & 0 \\
\vdots & \vdots & \vdots & \cdots & \vdots \\
- d^{(J)}_{N-1} & 0 & 0 & \cdots  & 1 \\
- d^{(J)}_N & 0 & 0 & \cdots & 0
\end{pmatrix} ~~~ : J \in \left\lbrace L,R \right\rbrace
\label{eq: Evolution Matrices Intermediate}
\end{equation}
with $d^{(J)}_i$ being the coefficient of $\lambda^i$ in the characteristic polynomial of $M_J$; $X=\left( x_1, \ldots, x_N \right)^T \in \mathbb{R}^N$ being a generic point in the phase space and $\zeta=\left(\mu, \ldots, 0 \right)^T$. Since the coefficient of $\lambda^i$ in the characteristic polynomial of a matrix is the sum of eigenvalues of the matrix taken $i$ at a time (apart from an alternating sign); we recast \eqref{eq: Evolution Matrices Intermediate} as
\begin{equation}
M_J =
\begin{pmatrix}
\rho^{(J)}_1 & 1 & 0 & \cdots  & 0 \\
- \rho^{(J)}_2 & 0 & 1 & \cdots & 0 \\
\vdots & \vdots & \vdots & \cdots & \vdots \\
(-1)^{N-2} ~ \rho^{(J)}_{N-1} & 0 & 0 & \cdots  & 1 \\
(-1)^{N-2} ~ \rho^{(J)}_N & 0 & 0 & \cdots & 0
\end{pmatrix} ~~~ : J \in \left\lbrace L,R \right\rbrace
\label{eq: Evolution Matrices}
\end{equation}
for further analysis. Here $\rho^{(J)}_i$ is the sum of the eigenvalues of $M_J$ taken $i$ at a time, i.e., $\rho^{(J)}_i=(-1)^{i-1} d^{(J)}_i$. For instance, in two and three dimensions, the matrix in \eqref{eq: Evolution Matrices} takes the form,
\begin{equation}
M_J = \begin{pmatrix}
\tau_J & 1 \\
-\delta_J & 0
\end{pmatrix} ~~~ : J \in \left\lbrace L,R \right\rbrace
\label{eq: Evolution Matrix 2D Initial}
\end{equation}
and
\begin{equation}
M_J =
\begin{pmatrix}
\tau_J & 1 & 0 \\
- \sigma_J & 0 & 1 \\
\delta_J & 0 & 0
\end{pmatrix} ~~~ : J \in \left\lbrace L,R \right\rbrace
\label{eq: Evolution Matrix 3D Initial}
\end{equation}
respectively. Here $\tau_J$ and $\delta_J$ are the trace and determinant the $M_J$; and $\sigma_J$ is the sum of the eigenvalues of $M_J$ taken two at a time.

Any $p$-periodic orbit in such a system can be represented by a sequence of points $\left\lbrace X_0, \ldots, X_i, \ldots, X_{p-1} \right\rbrace$ such that $X_{i+1}=G_\mu (X_i)$ where $i$ is a natural number or zero, and $X_{p}=X_0$. We can also associate a symbol to each point on the periodic orbit depending on the partition in which it lies. If a point of the periodic orbit has $x_1 \le 0$, it is assigned the symbol $L$ (meaning left). Otherwise it is assigned the symbol $R$ (meaning right). This associates a sequence of $R$ and $L$ to each periodic orbit. For example, an $LLLR$ (also written as $L^3R$) orbit consists of three points with $x_1 \le 0$ and a single point with $x_1>0$. To remove the ambiguity arising from possible cyclic permutation of the points, in this article we adopt the following convention: We number the points $X_i$ such that $X_0$ is assigned the symbol $R$ and $X_p$ is assigned the symbol $L$. However while referring of the orbits, we collect the symbols in the reverse order such that the symbol for the orbit starts with $L$ and ends with $R$. In this article, for sake of simplicity, we would restrict our explicit analysis to orbits of the form $L^mR^n$ where $m$ and $n$ are natural numbers. However, the theory developed here can be extended to the general finite period orbit $L^{n_1}R^{n_2} \ldots L^{n_{s-1}}R^{n_s}$. Ways to extend the theory for the general cases will be indicated wherever required. Note that since the map in each of the partitions is linear, a periodic orbit lying completely within any one of the partitions does not exist.

\section{Existence and Stability of Periodic Orbits}

In 1959, Leonov \cite{leonov1959map,leonov1962discontinuous} did a detailed study of nested period adding bifurcation structure occurring in piecewise-linear discontinuous 1D maps. The algorithmic way proposed in his work was recently used\cite{gardini2010border} to analyse border collision bifurcations in one dimensional maps. We extend the analysis to obtain the existence criteria for period orbits in $N$ dimensions.

Consider an $L^mR^n$ orbit formed by the points $\left\lbrace X_0, \ldots, X_{m+n-1} \right\rbrace$. According to the convention described in the last section, we assume the points $X_0, \ldots, X_{m-1}$ have $x_1 > 0$ and the points $X_m, \ldots, X_{m+n-1}$ have $x_1 \le 0$. Hence the evolution of $X_0$ under the map $G_\mu$ is given as,
\begin{align*}
X_1 &= M_R X_0 + \zeta \\
X_2 &= M_R^2 X_0 + \left( I + M_R \right) \zeta \\
X_3 &= M_R^3 X_0 + \left( I + M_R + M_R^2 \right) \zeta \\
\vdots & ~~~~~~~~~~	 \vdots \\
X_{n-1} &= M_R^{n-1} X_0 + \phi_{R,n-1} \zeta \\
X_n &= M_R^n X_0 + \phi_{R,n} \zeta \\
X_{n+1} &= M_L M_R^n X_0 + M_L \phi_{R,n} \zeta + \zeta \\
\vdots & ~~~~~~~~~~	 \vdots \\
X_{m+n-1} &= M_L^{m-1} M_R^n X_0 + \left( M_L^{m-1} \phi_{R,n} + \phi_{L,m-1} \right) \zeta \\
X_{m+n} &= M_L^m M_R^n X_0 + \left( M_L^m \phi_{R,n} + \phi_{L,m} \right) \zeta.
\end{align*}
Here $I$ is the $N \times N$ identity matrix and $\phi_{J,k} = I + M_J + M_J^2 + \ldots + M_J^{k-1}$. By  using the formula for GP of matrices, it can be written as, 
\begin{equation}
\phi_{J,k} = \left( I - M_J \right)^{-1} \left( I - M_J^k \right) ~~~ :J \in \left\lbrace L,R \right\rbrace 
\label{eq: phi Definition}
\end{equation}
if $I-M_J$ is invertible.

On substituting $X_{m+n}=X_0$, we get the expression for $X_0$ in terms of known quantities,
\begin{equation}
X_0 = \left( I - M_L^m M_R^n \right)^{-1} \left( M_L^m \phi_{R,n} + \phi_{L,m} \right) \zeta
\label{eq: Existence}
\end{equation}
when $\left( I - M_L^m M_R^n \right)$ is invertible. Once $X_0$ is determined, all other points of the orbit can be calculated by evolving $X_0$ under the map $G_\mu$. For existence of the $L^mR^n$ orbit, we need to ensure that all points of the periodic orbit are in their correct partitions. 


For periodic orbits to be stable, we require the trace $T$ and determinant $\Delta$ of  the ordered product of the Jacobian matrices at each point of the periodic orbit to follow
\begin{equation}
(1-\Delta) < T < (1+\Delta).
\label{eq: Stability}
\end{equation}
Since the Jacobian at any point with $x_1 \le 0$ is $M_L$ and that at any point with $x_1>0$ is $M_R$, the matrix whose trace and determinant are in question is $M_L^mM_R^n$.

For the generic case of $L^{n_1} R^{n_2} \ldots L^{n_{s-1}} R^{n_s}$ orbits, the algorithm of finding the existence and stability conditions of the orbits remain the same. On explicit calculation, the expression for the the point $X_0$ on the periodic orbit comes out to be
\begin{equation}
X_0 = \left( I - \prod_{j=1}^{s} M_K^{n_j} \right)^{-1} \left( \sum_{i=1}^{s} \left( \prod_{j=1}^{i-1} M_K^{n_j} \right) \phi_{K,n_i} \right) \zeta
\label{eq: Existence General}
\end{equation}
where $K$ is the symbol $L$ or $R$ when $j$ is odd or even respectively. The condition for stability remains identical to \eqref{eq: Stability}, with the only difference that now $T$ and $\Delta$ denote the trace and determinant of the matrix $\prod\limits_{j=1}^{s} M_K^{n_j}$.

A close look at equations \eqref{eq: Existence}, \eqref{eq: Stability} and \eqref{eq: Existence General} reveal that, the primary requirement in evaluating those expressions is to compute powers of $M_L$ and $M_R$. Typically this is done algorithmically by brute force matrix multiplication. The following sections aims to provide an algebraic way to compute the powers of the matrices and hence to analytically compute the regions of existence and stability of the periodic orbits.

\section{Computing $M^n$ for $N$ Dimensions}

In this section, we present a technique for computing $n^{th}$ powers of the matrix $M$, given by
\begin{equation}
M =
\begin{pmatrix}
\rho_1 & 1 & 0 & \cdots  & 0 \\
- \rho_2 & 0 & 1 & \cdots & 0 \\
\vdots & \vdots & \vdots & \cdots & \vdots \\
(-1)^{N-2} ~ \rho_{N-1} & 0 & 0 & \cdots  & 1 \\
(-1)^{N-1} ~ \rho_N & 0 & 0 & \cdots & 0
\end{pmatrix}.
\label{eq: M Definition}
\end{equation}

Note that this matrix is the same as the matrices $M_L$ or $M_R$ defined in \eqref{eq: Evolution Matrices} except that the subscripts $L$ and $R$ are ignored. This is done for notational simplicity. The results derived here are applicable to both $M_L$ and $M_R	$.

The $n^{th}$ power of any general $N \times N$ matrix
\begin{equation}
A = \begin{pmatrix}
\theta_{1,1} & \cdots & \theta_{1,N} \\
\vdots & \ddots & \vdots \\
\theta_{N,1} & \cdots & \theta_{N,N} \\
\end{pmatrix}
\end{equation}
can be written in the form
\begin{equation}
A^n = \begin{pmatrix}
\theta^{(n)}_{1,1} & \cdots & \theta^{(n)}_{1,N} \\
\vdots & \ddots & \vdots \\
\theta^{(n)}_{N,1} & \cdots & \theta^{(n)}_{N,N} \\
\end{pmatrix}
\end{equation}
The sequence of matrices $A, A^2, A^3 \cdots A^n$ is composed of sequences of its individual elements, and therefore $N^2$ such independent sequences are needed to construct the matrix $A^n$. However, we show that due to the structure of the matrix $M$ in \eqref{eq: M Definition},  only $N$ sequences are required to construct the matrix $M^n$.


Let each of the $N$ sequences be denoted by $\Gamma_i$. Let $\Gamma_{i,j}$ denote the $j^{th}$ term of the sequence $\Gamma_i$. Also note that for notational simplicity in the further analysis, we start numbering the terms of the sequence $\Gamma_i$ from $2-N$ instead of 1. For example, in order to compute $M^n$ in three dimensions, we would require three sequences: $\Gamma_1$, $\Gamma_2$ and $\Gamma_3$; and the terms of the sequences will be numbered as $\Gamma_{1,-1}, \Gamma_{1,0}, \Gamma_{1,1}, \Gamma_{1,2}, \cdots$ for the first sequence, $\Gamma_{2,-1}, \Gamma_{2,0}, \Gamma_{2,1}, \Gamma_{2,2}, \cdots$ for the second sequence and $\Gamma_{3,-1}, \Gamma_{3,0}, \Gamma_{3,1}, \Gamma_{3,2}, \cdots$ for the third sequence.

The terms of this sequence shall be defined iteratively. However, before giving the iterative relation, we first give the first $N$ terms of the sequences. This initialisation is done such that the $N$ elements of the $i^{th}$ row of $M$ give the first $N$ terms of $\Gamma_i$ in the reverse order, i.e.,
\begin{equation}
M=\begin{pmatrix}
\Gamma_{1,1} & \Gamma_{1,0} & \cdots & \Gamma_{1,2-N} \\
\Gamma_{2,1} & \Gamma_{2,0} & \cdots & \Gamma_{2,2-N} \\
\vdots & \vdots & \cdots & \vdots \\
\Gamma_{N,1} & \Gamma_{N,0} & \cdots & \Gamma_{N,2-N}
\end{pmatrix}.
\label{eq: Initialisation}
\end{equation}
Again taking the three dimensional case as an example, we would have the initialisation as,
\begin{align}
\Gamma_{1,-1} = 0, ~ \Gamma_{1,0} = 1, \Gamma_{1,1} = \tau \nonumber \\
\Gamma_{2,-1} = 1, ~ \Gamma_{2,0} = 0, \Gamma_{2,1} = -\sigma \label{eq: Example} \\
\Gamma_{3,-1} = 0, ~ \Gamma_{3,0} = 0, \Gamma_{3,1} = \delta. \nonumber
\end{align}

Note the due to the initialisation given above, the terms from $\Gamma_{i,2-N}$ to $\Gamma_{i,1}$ are defined. Now, the further terms of the sequence are defined iteratively as
\begin{equation}
\Gamma_{i,j} = \sum_{k=1}^N (-1)^{k-1} ~ \rho_k ~ \Gamma_{i,j-k} ~~~ : j \in [2, \infty).
\label{eq: Iteration}
\end{equation}


With this definition of $\Gamma_{i,j}$, the matrix $M^n$ can be written as
\begin{equation}
M^n =
\begin{pmatrix}
\Gamma_{1,n} & \Gamma_{1,n-1} & \ldots & \Gamma_{1,n-(N-1)} \\
\Gamma_{2,n} & ~ & ~ & \Gamma_{2,n-(N-1)} \\
\vdots & ~ & ~ & \vdots \\
\Gamma_{N,n} & \Gamma_{N,n-1} & \ldots & \Gamma_{N,n-(N-1)}
\end{pmatrix}
\label{eq: General Result}
\end{equation}
or more explicitly as
\begin{equation}
\left[ M^n \right]_{i,j}=\Gamma_{i,n-(j-1)} ~~~ \forall n \in \mathbb{N}.
\label{eq: General Result Elementwise}
\end{equation}

The proof of the result hinges on the structure of the matrix $M$. Note that \eqref{eq: General Result Elementwise} is bound to be satisfied for $n=1$ due to the way the series $\Gamma_i$ is initialised in \eqref{eq: Initialisation}. Also note that, apart from the first column, the only non-zero elements of the matrix $M$ are in the superdiagonal and are all equal to 1. Hence, while obtaining $M^{k+1}$ from $M^k$, when $M$ is multiplied from the right, the columns of $M^k$ are simply shifted to right, except the rightmost column which is lost. Therefore, only the first column needs to be computed, which is given by \eqref{eq: Iteration}. The detailed proof of the result is given in Appendix \ref{app: Generic}.

Hence in order to compute $M^{k+1}$ from $M^k$, one needs to compute only $N$ new elements corresponding to the first column of the matrix; as compared to computing $N^2$ new elements for a generic $N \times N$ matrix. 

Having described the algorithm of finding the $n^{th}$ power of an $N$ dimensional matrix, we apply the algorithm to 3 dimensional piecewise linear maps to compute the regions in parameter space where stable periodic orbits might exist.

\section{Application: Stable $L^nR$ Orbits in 3 Dimensions}

In 2012, parameter regions where stable $L^nR$ orbits exist in two dimensions were computed to find the parameter regions of multiple attractor bifurcations \cite{avrutin2012occurrence}. In this section, we extend the result to demonstrate the use of the technique developed in this article to find the regions of existence of stable $L^nR$ periodic orbits in 3 dimensions. As the parameter space is 6 dimensional, we show the two dimensional projection of the plausible regions in parameter space.

\begin{figure}
\includegraphics[width=0.45\textwidth]{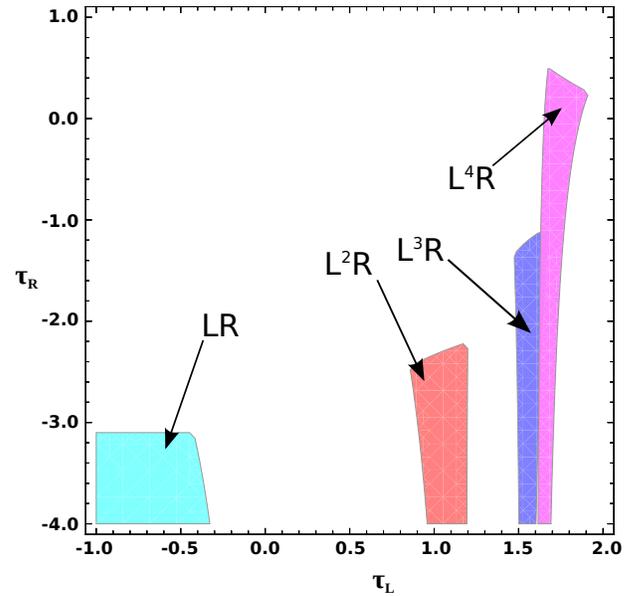}
\caption{Regions of the projected parameter space where stable $L^nR$ orbits exist. The plot has been made by setting $\sigma_L=\sigma_R=1.4$, $\delta_L=\delta_R=0.7$ and $\mu=1$.}
\label{fig: Parameter Plot}
\end{figure}

In three dimensions, the normal matrix \eqref{eq: M Definition} takes the form
\begin{equation}
M =
\begin{pmatrix}
\tau & 1 & 0 \\
- \sigma & 0 & 1 \\
\delta & 0 & 0
\end{pmatrix}.
\end{equation}
Hence, by \eqref{eq: General Result}
\begin{equation}
M^n
=
\begin{pmatrix}
\Gamma_{1,n} & \Gamma_{1,n-1} & \Gamma_{1,n-2} \\
\Gamma_{2,n} & \Gamma_{2,n-1} & \Gamma_{2,n-2} \\
\Gamma_{3,n} & \Gamma_{3,n-1} & \Gamma_{3,n-2}
\end{pmatrix}
\end{equation}
where
\begin{equation}
\Gamma_{i,n} = \tau \Gamma_{i,n-1} - \sigma \Gamma_{i,n-2} + \delta \Gamma_{i,n-3}
\end{equation}
for $i \in \left\lbrace 1,2,3 \right\rbrace$ and $n>1$. For $-1 \le n \le 1$, the terms are taken from $M$ according to \eqref{eq: Example}. Using these expressions, we compute the required powers of $M_L$ and $M_R$, and substitute them in \eqref{eq: Existence} and \eqref{eq: Stability} to obtain the required conditions for existence and stability of $L^nR$ orbits for various values of $n$. The parameter values for which the stable orbits exist are shown in Fig.~\ref{fig: Parameter Plot} where the plausible regions are shown in the two dimensional projection of the six dimensional space.

Now let us consider the special case of 2 dimensions where further simplification can be done and the $n^{th}$ power of $M$ can be computed non-iteratively.

\section{Computing $M^n$ for 2 Dimensions}

In two dimensions the matrix appearing in the normal form map is given by \eqref{eq: Evolution Matrix 2D Initial}.
 Dropping the subscripts for notational simplicity as in the previous section gives us
\begin{equation}
M = \begin{pmatrix}
\tau & 1 \\
-\delta & 0
\end{pmatrix}.
\label{eq: M Definition 2D}
\end{equation}

From the results obtained in the previous section, it can be said that $M^n$ would be determined by two independent sequences $a$ and $b$ and would be of the form
\begin{equation}
M^n = \begin{pmatrix}
a_n & a_{n-1} \\
b_n & b_{n-1}
\end{pmatrix}.
\label{eq: Result 2D Intermediate}
\end{equation}

However on explicit calculation (as done in Appendix \ref{app: Motivation}) it can be shown that the terms of the two sequences are related as
\begin{equation}
b_i = -\delta a_{i-1} ~~~ \forall i \geq 0.
\label{eq: an bn relation}
\end{equation}

Substituting \eqref{eq: an bn relation} in \eqref{eq: Result 2D Intermediate} gives the final form of $M^n$ as
\begin{equation}
M^n=\begin{pmatrix}
a_n & a_{n-1} \\
-\delta a_{n-1} & -\delta a_{n-2}
\end{pmatrix} ~~~ \forall n \in \mathbb{N}
\label{eq: M^n Form 2D}
\end{equation}
where
\begin{equation}
a_n = \tau a_{n-1} - \delta a_{n-2} ~~~ \forall i \in \mathbb{N}
\label{eq: Propagation Rule 2D}
\end{equation}
with initial conditions $a_0=1$ and $a_{-1}=0$. Apart from the iterative definition, it is also possible to explicitly determine $a_n$ in terms of known quantities.

In terms of $\tau$ and $\delta$, $a_n$ is given as
\begin{equation}
a_n = \sum\limits_{m=0}^{\left[\dfrac{n}{2}\right]} \left( -1 \right)^m ~^{n-m}C_m~ \delta^m ~ \tau^{n-2m} ~~~ \forall n \ge 0
\label{eq: a_n Parameters}
\end{equation}
and $a_{-1}=0$. Here $[\cdot]$ is the greatest integer function and $~^nC_r$ is the coefficient of $x^r$ in the binomial expansion of $(1+x)^n$. Using the properties of $~^nC_r$, it can be shown the $a_n$ as expressed in \eqref{eq: a_n Parameters} satisfies \eqref{eq: Propagation Rule 2D}. However as the proof is lengthy, it is given in Appendix \ref{app: Fundamental Sequence}.

We can also obtain another representation of $a_n$ if the results are expressed in terms of the eigenvalues of $M$,
\begin{equation}
\lambda_{1,2} = \dfrac{\tau \pm \sqrt{\tau^2-4\delta}}{2}.
\end{equation}
To do so, we put $\tau=\lambda_1+\lambda_2$ and $\delta=\lambda_1\lambda_2$ in \eqref{eq: M Definition 2D}. We then decompose $M$ as $M=UDU^{-1}$ where $D$ is the diagonal matrix with $\lambda_1$ and $\lambda_2$ on its diagonals and $U$ is the matrix with the eigenvectors of $M$ as the columns. Then $M^n=UD^nU^{-1}$, which when computed explicitly gives
\begin{equation*}
M^n=\dfrac{1}{\lambda_2 - \lambda_1}
\begin{pmatrix}
\lambda_2^{n+1} - \lambda_1^{n+1} & \lambda_2^n - \lambda_1^n \\
\lambda_1^{n+1}\lambda_2 - \lambda_2^{n+1}\lambda_1 & \lambda_1^n\lambda_2 - \lambda_2^n\lambda_1
\end{pmatrix}
\end{equation*}
which can be recast into the form of \eqref{eq: M^n Form 2D} with the definition of $a_n$ as
\begin{equation}
a_n = \dfrac{\lambda_1^{n+1}-\lambda_2^{n+1}}{\lambda_1-\lambda_2}.
\label{eq: a_n Eigen}
\end{equation}
The detailed proof of this result is given in the Appendix \ref{app: Eigen}. It can also be directly seen that \eqref{eq: a_n Eigen} satisfies \eqref{eq: Propagation Rule 2D}.

Hence in order to compute $M^n$ for a $2 \times 2$ matrix in the form \eqref{eq: M Definition 2D}, we use \eqref{eq: a_n Parameters} or \eqref{eq: a_n Eigen} to get $a_n$, $a_{n-1}$, $a_{n-2}$ and form the matrix given in \eqref{eq: M^n Form 2D}.

The form of the matrix $M^n$ in \eqref{eq: M^n Form 2D} may be substituted directly into \eqref{eq: phi Definition} to obtain the sum of the GP as
\begin{equation}
\phi_n = \begin{pmatrix}
f_n & f_{n-1} \\
- \delta f_{n-1} & 1 - \delta f_{n-2}
\end{pmatrix}
\label{eq: phi_n Form}
\end{equation}
where
\begin{equation}
f_n = \dfrac{1 - a_n + \delta a_{n-1}}{1 - \tau + \delta}.
\label{eq: f_n Form}
\end{equation}
The detailed proof of this result is given in Appendix \ref{app: phi}. Note that the proof is independent of the explicit form of $a_n$ and uses only the propagation rule \eqref{eq: Propagation Rule 2D}.

\section{Comments on Speed of the Algorithm}

Although results in Section-5 demonstrate the use of the algorithm described in this article, the effectiveness of the proposed algorithm can be gauged better in higher dimensional systems where computation of existence and stability conditions becomes computationally intensive due to large orders of matrices involved. In this section, we show the efficiency of the proposed method against the traditional methods as the dimensions of the system increases. To do this we compare the times required by the proposed method to compute powers of the matrix $M_{L/R}$ for various dimensions against the times required by other frequently used methods for the same computation. 

First let us look at the computational complexity of the proposed algorithm. As noted erlier, the efficiency of the proposed method lies in the structure of the matrices $M_{L/R}$. We note from \eqref{eq: General Result Elementwise} that
\begin{equation}
\left[ M^n \right]_{i,j}=\Gamma_{i,n-(j-1)} = \Gamma_{i,(n+1)-((j+1)-1)}=\left[ M^{n+1} \right]_{i,j+1}.
\end{equation}
This implies that in $N$-dimensions, the first $N-1$ columns of $M^n$ is identical to the last $N-1$ columns of $M^{n+1}$. Hence, to compute $M^{n+1}$ from $M^n$, we need to compute only one new column of the matrix. The construction of this new column is described by \eqref{eq: Iteration}. Computation of each term of the column requires $N$ multiplications. As there are $N$ elements in the row, the total number of multiplications to be performed to raise the $N$-dimensional matrix $M$ to the power $n$ is $N^2n$. Hence, the proposed algorithm has the computational complexity of the order $\mathcal{O}(N^2n)$.

Two of the most common methods for computing the powers of matrices are brute-force matrix multiplication and matrix diagonalisation. Brute-force matrix multiplication simply multiplies the matrices successively; hence making computation of $n^{th}$ power of an $N$-dimensional matrix an $\mathcal{O}(N^3n)$ process. On the other hand, matrix diagonalisation seeks to diagonalise the matrix $M$ as $M=UDU^{-1}$ and then compute the $n^{th}$ power as $M^n=UD^nU^{-1}$. While the method seems elegant, it involves diagonalisation of the matrix; making it an $\mathcal{O}(N^3)$ process \cite{arfken2005mathematical}. Even with more sophisticated matrix multiplications like Strassen algorithm~\cite{strassen1969gaussian} or Coppersmith-Winograd algorithm~\cite{le2014powers,davie2013improved,coppersmith1997rectangular}, one obtains a order complexity of $\mathcal{O}(N^qn)$ with $q>2$. Comparing the order complexity of the proposed method (which is $\mathcal{O}(N^2n)$) with that of brute force matrix multiplication or matrix diagonalisation, we see that the proposed algorithm is better than the other methods at least for large enough matrix dimension.

\begin{figure}
\begin{subfigure}{0.45\textwidth}
\includegraphics[width=\textwidth]{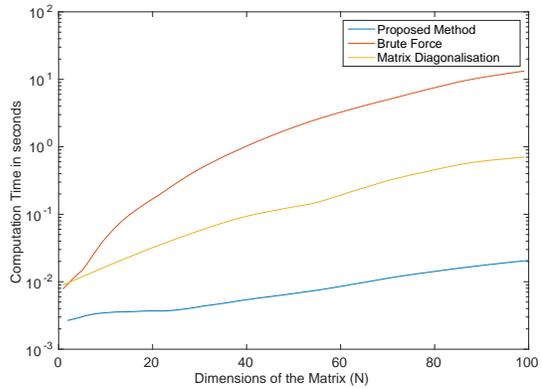}
\caption{Power computed, $m=10$}
\end{subfigure}
\begin{subfigure}{0.45\textwidth}
\includegraphics[width=\textwidth]{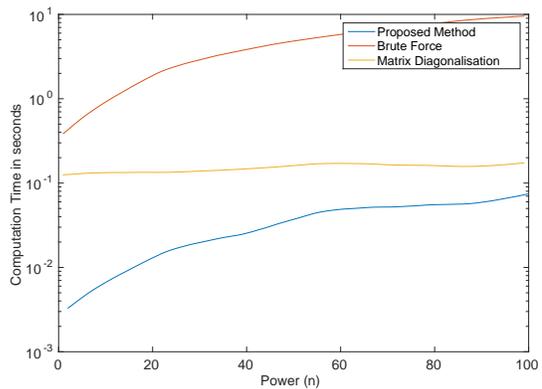}
\caption{Dimension of the Matrix, $N=50$}
\end{subfigure}
\caption{Comparison of the proposed algorithm of raising matrices to powers with the pre-existing methods. Plots show the time taken to perform 100 typical calculations by each of the algorithms. For each dimension and power, 1000 sets of 100 random matrices of the form \eqref{eq: M Definition} each were taken and the total time taken for computation was averaged over the 1000 sets.}
\label{fig: Comparison}
\end{figure}

However, to get an estimate for realistic dimensions and powers, we computed the powers of random matrices of the form \eqref{eq: M Definition} on a computer using the three algorithms: a) the proposed algorithm, b) matrix diagonalisation and c) brute force matrix multiplication to compare their run-times. Figure~\ref{fig: Comparison} shows the time taken by the different algorithms on the same computer to compute the powers of random matrices of the form \eqref{eq: M Definition}. As can be evidently seen, the efficiency of the proposed method, when compared to the other algorithms increases with increase in the dimension of the matrices when the power to which the matrices are raised is kept constant. When the matrix power is increased keeping the dimension of the matrix fixed, we see that the gain provided by the proposed algorithm as compared to matrix diagonalisation reduces as $n$ is increased. However, the proposed algorithm performs better than matrix diagonalisation for a significantly large range of $n$. Only for very high powers of matrices is matrix diagonaisation a better algorithm than the proposed one.

Hence the proposed algorithm is seen to perform better than the competing algorithms for high dimensional systems when we are interested in sufficiently lower powers of the matrix. Physically, this corresponds to periodic orbits of sufficiently low periodicity in high dimensional phase space. Non-smooth dynamical systems with high dimensionality occur in many real-life electrical, electronic and robotic systems where the dimensions of the phase space can well go over 50 due the presence of many components. The proposed algorithm might hence be applied to such dynamical systems to increase computational efficiency.

\section{Conclusion}

In this article we developed a faster and a more elegant technique to compute the existence and stability conditions for periodic orbits of the form $L^mR^n$ in an arbitrary dimensional piecewise linear continuous map. The technique is based on easier computation of powers of $N \times N$ matrices in their normal form. Due to the structure of the matrices involved, it was found that the elements of the resulting matrix were interrelated; and in order to compute the $n^{th}$ power of the matrix, only $N$ out of the $N^2$ elements need to be computed. These $N$ elements in turn can be obtained as simple sequences defined iteratively. Once the powers of the matrices are computed, they can be substituted in the generic expressions of existence and stability of orbits to obtain the regions in parameter space where they exist. We also apply the technique developed to 3 dimensional systems and obtain the regions where $L^nR$ orbits exist.

Moreover, in the special case of 2 dimensional matrices, further simplifications were made. Notably, explicit expressions for the terms of the sequence in terms of the given parameters and eigenvalues of the matrices were obtained. This allows for a direct evaluation of any power of the $2 \times 2$ normal form matrix without computing the intermediate powers.


\section*{Acknowledgements}

The authors would like to thank Viktor Avrutin for constructive suggestions on the earlier versions of the article. A. S. would also like to thank Matthias Schr\"oder for fruitful discussions on the computational complexity of the method.

\bibliography{NonSmooth}

\begin{thebibliography}{25}%
\makeatletter
\providecommand \@ifxundefined [1]{%
 \@ifx{#1\undefined}
}%
\providecommand \@ifnum [1]{%
 \ifnum #1\expandafter \@firstoftwo
 \else \expandafter \@secondoftwo
 \fi
}%
\providecommand \@ifx [1]{%
 \ifx #1\expandafter \@firstoftwo
 \else \expandafter \@secondoftwo
 \fi
}%
\providecommand \natexlab [1]{#1}%
\providecommand \enquote  [1]{``#1''}%
\providecommand \bibnamefont  [1]{#1}%
\providecommand \bibfnamefont [1]{#1}%
\providecommand \citenamefont [1]{#1}%
\providecommand \href@noop [0]{\@secondoftwo}%
\providecommand \href [0]{\begingroup \@sanitize@url \@href}%
\providecommand \@href[1]{\@@startlink{#1}\@@href}%
\providecommand \@@href[1]{\endgroup#1\@@endlink}%
\providecommand \@sanitize@url [0]{\catcode `\\12\catcode `\$12\catcode
  `\&12\catcode `\#12\catcode `\^12\catcode `\_12\catcode `\%12\relax}%
\providecommand \@@startlink[1]{}%
\providecommand \@@endlink[0]{}%
\providecommand \url  [0]{\begingroup\@sanitize@url \@url }%
\providecommand \@url [1]{\endgroup\@href {#1}{\urlprefix }}%
\providecommand \urlprefix  [0]{URL }%
\providecommand \Eprint [0]{\href }%
\providecommand \doibase [0]{http://dx.doi.org/}%
\providecommand \selectlanguage [0]{\@gobble}%
\providecommand \bibinfo  [0]{\@secondoftwo}%
\providecommand \bibfield  [0]{\@secondoftwo}%
\providecommand \translation [1]{[#1]}%
\providecommand \BibitemOpen [0]{}%
\providecommand \bibitemStop [0]{}%
\providecommand \bibitemNoStop [0]{.\EOS\space}%
\providecommand \EOS [0]{\spacefactor3000\relax}%
\providecommand \BibitemShut  [1]{\csname bibitem#1\endcsname}%
\let\auto@bib@innerbib\@empty
\bibitem [{\citenamefont {Deane}\ and\ \citenamefont
  {Hamill}(1990)}]{deane1990instability}%
  \BibitemOpen
  \bibfield  {author} {\bibinfo {author} {\bibfnamefont {J.~H.}\ \bibnamefont
  {Deane}}\ and\ \bibinfo {author} {\bibfnamefont {D.~C.}\ \bibnamefont
  {Hamill}},\ }\href@noop {} {\bibfield  {journal} {\bibinfo  {journal} {Power
  Electronics, IEEE Transactions on}\ }\textbf {\bibinfo {volume} {5}},\
  \bibinfo {pages} {260} (\bibinfo {year} {1990})}\BibitemShut {NoStop}%
\bibitem [{\citenamefont {Kousaka}\ \emph {et~al.}(1999)\citenamefont
  {Kousaka}, \citenamefont {Ueta},\ and\ \citenamefont
  {Kawakami}}]{kousaka1999bifurcation}%
  \BibitemOpen
  \bibfield  {author} {\bibinfo {author} {\bibfnamefont {T.}~\bibnamefont
  {Kousaka}}, \bibinfo {author} {\bibfnamefont {T.}~\bibnamefont {Ueta}}, \
  and\ \bibinfo {author} {\bibfnamefont {H.}~\bibnamefont {Kawakami}},\
  }\href@noop {} {\bibfield  {journal} {\bibinfo  {journal} {Circuits and
  Systems II: Analog and Digital Signal Processing, IEEE Transactions on}\
  }\textbf {\bibinfo {volume} {46}},\ \bibinfo {pages} {878} (\bibinfo {year}
  {1999})}\BibitemShut {NoStop}%
\bibitem [{\citenamefont {Nordmark}(1991)}]{nordmark1991non}%
  \BibitemOpen
  \bibfield  {author} {\bibinfo {author} {\bibfnamefont {A.~B.}\ \bibnamefont
  {Nordmark}},\ }\href@noop {} {\bibfield  {journal} {\bibinfo  {journal}
  {Journal of Sound and Vibration}\ }\textbf {\bibinfo {volume} {145}},\
  \bibinfo {pages} {279} (\bibinfo {year} {1991})}\BibitemShut {NoStop}%
\bibitem [{\citenamefont {Thuilot}\ \emph {et~al.}(1997)\citenamefont
  {Thuilot}, \citenamefont {Goswami},\ and\ \citenamefont
  {Espiau}}]{thuilot1997bifurcation}%
  \BibitemOpen
  \bibfield  {author} {\bibinfo {author} {\bibfnamefont {B.}~\bibnamefont
  {Thuilot}}, \bibinfo {author} {\bibfnamefont {A.}~\bibnamefont {Goswami}}, \
  and\ \bibinfo {author} {\bibfnamefont {B.}~\bibnamefont {Espiau}},\ }in\
  \href@noop {} {\emph {\bibinfo {booktitle} {Robotics and Automation, 1997.
  Proceedings., 1997 IEEE International Conference on}}},\ Vol.~\bibinfo
  {volume} {1}\ (\bibinfo {organization} {IEEE},\ \bibinfo {year} {1997})\ pp.\
  \bibinfo {pages} {792--798}\BibitemShut {NoStop}%
\bibitem [{\citenamefont {Garcia}\ \emph {et~al.}(1998)\citenamefont {Garcia},
  \citenamefont {Chatterjee}, \citenamefont {Ruina},\ and\ \citenamefont
  {Coleman}}]{garcia1998simplest}%
  \BibitemOpen
  \bibfield  {author} {\bibinfo {author} {\bibfnamefont {M.}~\bibnamefont
  {Garcia}}, \bibinfo {author} {\bibfnamefont {A.}~\bibnamefont {Chatterjee}},
  \bibinfo {author} {\bibfnamefont {A.}~\bibnamefont {Ruina}}, \ and\ \bibinfo
  {author} {\bibfnamefont {M.}~\bibnamefont {Coleman}},\ }\href@noop {}
  {\bibfield  {journal} {\bibinfo  {journal} {Journal of biomechanical
  engineering}\ }\textbf {\bibinfo {volume} {120}},\ \bibinfo {pages} {281}
  (\bibinfo {year} {1998})}\BibitemShut {NoStop}%
\bibitem [{\citenamefont {Sun}\ \emph {et~al.}(1995)\citenamefont {Sun},
  \citenamefont {Amellal}, \citenamefont {Glass},\ and\ \citenamefont
  {Billette}}]{sun1995alternans}%
  \BibitemOpen
  \bibfield  {author} {\bibinfo {author} {\bibfnamefont {J.}~\bibnamefont
  {Sun}}, \bibinfo {author} {\bibfnamefont {F.}~\bibnamefont {Amellal}},
  \bibinfo {author} {\bibfnamefont {L.}~\bibnamefont {Glass}}, \ and\ \bibinfo
  {author} {\bibfnamefont {J.}~\bibnamefont {Billette}},\ }\href@noop {}
  {\bibfield  {journal} {\bibinfo  {journal} {Journal of theoretical biology}\
  }\textbf {\bibinfo {volume} {173}},\ \bibinfo {pages} {79} (\bibinfo {year}
  {1995})}\BibitemShut {NoStop}%
\bibitem [{\citenamefont {B{\"o}rgers}\ and\ \citenamefont
  {Kopell}(2003)}]{borgers2003synchronization}%
  \BibitemOpen
  \bibfield  {author} {\bibinfo {author} {\bibfnamefont {C.}~\bibnamefont
  {B{\"o}rgers}}\ and\ \bibinfo {author} {\bibfnamefont {N.}~\bibnamefont
  {Kopell}},\ }\href@noop {} {\bibfield  {journal} {\bibinfo  {journal} {Neural
  computation}\ }\textbf {\bibinfo {volume} {15}},\ \bibinfo {pages} {509}
  (\bibinfo {year} {2003})}\BibitemShut {NoStop}%
\bibitem [{\citenamefont {Banerjee}\ \emph {et~al.}(1998)\citenamefont
  {Banerjee}, \citenamefont {Yorke},\ and\ \citenamefont
  {Grebogi}}]{banerjee1998robust}%
  \BibitemOpen
  \bibfield  {author} {\bibinfo {author} {\bibfnamefont {S.}~\bibnamefont
  {Banerjee}}, \bibinfo {author} {\bibfnamefont {J.~A.}\ \bibnamefont {Yorke}},
  \ and\ \bibinfo {author} {\bibfnamefont {C.}~\bibnamefont {Grebogi}},\
  }\href@noop {} {\bibfield  {journal} {\bibinfo  {journal} {Physical Review
  Letters}\ }\textbf {\bibinfo {volume} {80}},\ \bibinfo {pages} {3049}
  (\bibinfo {year} {1998})}\BibitemShut {NoStop}%
\bibitem [{\citenamefont {Simpson}(2014)}]{simpson2014sequences}%
  \BibitemOpen
  \bibfield  {author} {\bibinfo {author} {\bibfnamefont {D.~J.}\ \bibnamefont
  {Simpson}},\ }\href@noop {} {\bibfield  {journal} {\bibinfo  {journal}
  {International Journal of Bifurcation and Chaos}\ }\textbf {\bibinfo {volume}
  {24}} (\bibinfo {year} {2014})}\BibitemShut {NoStop}%
\bibitem [{\citenamefont {Nusse}\ and\ \citenamefont
  {Yorke}(1992)}]{nusse1992border}%
  \BibitemOpen
  \bibfield  {author} {\bibinfo {author} {\bibfnamefont {H.~E.}\ \bibnamefont
  {Nusse}}\ and\ \bibinfo {author} {\bibfnamefont {J.~A.}\ \bibnamefont
  {Yorke}},\ }\href@noop {} {\bibfield  {journal} {\bibinfo  {journal} {Physica
  D: Nonlinear Phenomena}\ }\textbf {\bibinfo {volume} {57}},\ \bibinfo {pages}
  {39} (\bibinfo {year} {1992})}\BibitemShut {NoStop}%
\bibitem [{\citenamefont {Roy}\ and\ \citenamefont
  {Roy}(2008)}]{roy2008border}%
  \BibitemOpen
  \bibfield  {author} {\bibinfo {author} {\bibfnamefont {I.}~\bibnamefont
  {Roy}}\ and\ \bibinfo {author} {\bibfnamefont {A.}~\bibnamefont {Roy}},\
  }\href@noop {} {\bibfield  {journal} {\bibinfo  {journal} {International
  Journal of Bifurcation and Chaos}\ }\textbf {\bibinfo {volume} {18}},\
  \bibinfo {pages} {577} (\bibinfo {year} {2008})}\BibitemShut {NoStop}%
\bibitem [{\citenamefont {Avrutin}\ \emph {et~al.}(2012)\citenamefont
  {Avrutin}, \citenamefont {Schanz},\ and\ \citenamefont
  {Banerjee}}]{avrutin2012occurrence}%
  \BibitemOpen
  \bibfield  {author} {\bibinfo {author} {\bibfnamefont {V.}~\bibnamefont
  {Avrutin}}, \bibinfo {author} {\bibfnamefont {M.}~\bibnamefont {Schanz}}, \
  and\ \bibinfo {author} {\bibfnamefont {S.}~\bibnamefont {Banerjee}},\
  }\href@noop {} {\bibfield  {journal} {\bibinfo  {journal} {Nonlinear
  Dynamics}\ }\textbf {\bibinfo {volume} {67}},\ \bibinfo {pages} {293}
  (\bibinfo {year} {2012})}\BibitemShut {NoStop}%
\bibitem [{\citenamefont {Gardini}\ \emph {et~al.}(2010)\citenamefont
  {Gardini}, \citenamefont {Tramontana}, \citenamefont {Avrutin},\ and\
  \citenamefont {Schanz}}]{gardini2010border}%
  \BibitemOpen
  \bibfield  {author} {\bibinfo {author} {\bibfnamefont {L.}~\bibnamefont
  {Gardini}}, \bibinfo {author} {\bibfnamefont {F.}~\bibnamefont {Tramontana}},
  \bibinfo {author} {\bibfnamefont {V.}~\bibnamefont {Avrutin}}, \ and\
  \bibinfo {author} {\bibfnamefont {M.}~\bibnamefont {Schanz}},\ }\href@noop {}
  {\bibfield  {journal} {\bibinfo  {journal} {International Journal of
  Bifurcation and Chaos}\ }\textbf {\bibinfo {volume} {20}},\ \bibinfo {pages}
  {3085} (\bibinfo {year} {2010})}\BibitemShut {NoStop}%
\bibitem [{\citenamefont {Ganguli}\ and\ \citenamefont
  {Banerjee}(2005)}]{ganguli2005dangerous}%
  \BibitemOpen
  \bibfield  {author} {\bibinfo {author} {\bibfnamefont {A.}~\bibnamefont
  {Ganguli}}\ and\ \bibinfo {author} {\bibfnamefont {S.}~\bibnamefont
  {Banerjee}},\ }\href@noop {} {\bibfield  {journal} {\bibinfo  {journal}
  {Physical Review E}\ }\textbf {\bibinfo {volume} {71}},\ \bibinfo {pages}
  {057202\_1} (\bibinfo {year} {2005})}\BibitemShut {NoStop}%
\bibitem [{\citenamefont {Panchuk}\ \emph {et~al.}(2015)\citenamefont
  {Panchuk}, \citenamefont {Sushko},\ and\ \citenamefont
  {Avrutin}}]{panchuk2015bifurcation}%
  \BibitemOpen
  \bibfield  {author} {\bibinfo {author} {\bibfnamefont {A.}~\bibnamefont
  {Panchuk}}, \bibinfo {author} {\bibfnamefont {I.}~\bibnamefont {Sushko}}, \
  and\ \bibinfo {author} {\bibfnamefont {V.}~\bibnamefont {Avrutin}},\
  }\href@noop {} {\bibfield  {journal} {\bibinfo  {journal} {International
  Journal of Bifurcation and Chaos}\ }\textbf {\bibinfo {volume} {25}}
  (\bibinfo {year} {2015})}\BibitemShut {NoStop}%
\bibitem [{\citenamefont {Leonov}(1959)}]{leonov1959map}%
  \BibitemOpen
  \bibfield  {author} {\bibinfo {author} {\bibfnamefont {N.}~\bibnamefont
  {Leonov}},\ }\href@noop {} {\bibfield  {journal} {\bibinfo  {journal}
  {Radiofisica}\ }\textbf {\bibinfo {volume} {3}},\ \bibinfo {pages} {942}
  (\bibinfo {year} {1959})}\BibitemShut {NoStop}%
\bibitem [{\citenamefont {Leonov}(1962)}]{leonov1962discontinuous}%
  \BibitemOpen
  \bibfield  {author} {\bibinfo {author} {\bibfnamefont {N.}~\bibnamefont
  {Leonov}},\ }\href@noop {} {\bibfield  {journal} {\bibinfo  {journal}
  {Doklady Akademii Nauk SSSR}\ }\textbf {\bibinfo {volume} {143}},\ \bibinfo
  {pages} {1038} (\bibinfo {year} {1962})}\BibitemShut {NoStop}%
\bibitem [{\citenamefont {Avrutin}\ \emph {et~al.}(2010)\citenamefont
  {Avrutin}, \citenamefont {Schanz},\ and\ \citenamefont
  {Gardini}}]{avrutin2010self}%
  \BibitemOpen
  \bibfield  {author} {\bibinfo {author} {\bibfnamefont {V.}~\bibnamefont
  {Avrutin}}, \bibinfo {author} {\bibfnamefont {M.}~\bibnamefont {Schanz}}, \
  and\ \bibinfo {author} {\bibfnamefont {L.}~\bibnamefont {Gardini}},\
  }\href@noop {} {\bibfield  {journal} {\bibinfo  {journal} {Regular and
  Chaotic Dynamics}\ }\textbf {\bibinfo {volume} {15}},\ \bibinfo {pages} {685}
  (\bibinfo {year} {2010})}\BibitemShut {NoStop}%
\bibitem [{\citenamefont {Tramontana}\ \emph {et~al.}(2012)\citenamefont
  {Tramontana}, \citenamefont {Gardini}, \citenamefont {Avrutin},\ and\
  \citenamefont {Schanz}}]{tramontana2012period}%
  \BibitemOpen
  \bibfield  {author} {\bibinfo {author} {\bibfnamefont {F.}~\bibnamefont
  {Tramontana}}, \bibinfo {author} {\bibfnamefont {L.}~\bibnamefont {Gardini}},
  \bibinfo {author} {\bibfnamefont {V.}~\bibnamefont {Avrutin}}, \ and\
  \bibinfo {author} {\bibfnamefont {M.}~\bibnamefont {Schanz}},\ }\href@noop {}
  {\bibfield  {journal} {\bibinfo  {journal} {International Journal of
  Bifurcation and Chaos}\ }\textbf {\bibinfo {volume} {22}} (\bibinfo {year}
  {2012})}\BibitemShut {NoStop}%
\bibitem [{\citenamefont {di~Bernardo}(2003)}]{di2003normal}%
  \BibitemOpen
  \bibfield  {author} {\bibinfo {author} {\bibfnamefont {M.}~\bibnamefont
  {di~Bernardo}},\ }in\ \href@noop {} {\emph {\bibinfo {booktitle} {Circuits
  and Systems, 2003. ISCAS'03. Proceedings of the 2003 International Symposium
  on}}},\ Vol.~\bibinfo {volume} {3}\ (\bibinfo {organization} {IEEE},\
  \bibinfo {year} {2003})\ pp.\ \bibinfo {pages} {III--76}\BibitemShut
  {NoStop}%
\bibitem [{\citenamefont {Arfken}\ and\ \citenamefont
  {Weber}(2005)}]{arfken2005mathematical}%
  \BibitemOpen
  \bibfield  {author} {\bibinfo {author} {\bibfnamefont {G.~B.}\ \bibnamefont
  {Arfken}}\ and\ \bibinfo {author} {\bibfnamefont {H.~J.}\ \bibnamefont
  {Weber}},\ }\href@noop {} {\emph {\bibinfo {title} {Mathematical Methods For
  Physicists International Student Edition}}}\ (\bibinfo  {publisher} {Academic
  press},\ \bibinfo {year} {2005})\BibitemShut {NoStop}%
\bibitem [{\citenamefont {Strassen}(1969)}]{strassen1969gaussian}%
  \BibitemOpen
  \bibfield  {author} {\bibinfo {author} {\bibfnamefont {V.}~\bibnamefont
  {Strassen}},\ }\href@noop {} {\bibfield  {journal} {\bibinfo  {journal}
  {Numerische Mathematik}\ }\textbf {\bibinfo {volume} {13}},\ \bibinfo {pages}
  {354} (\bibinfo {year} {1969})}\BibitemShut {NoStop}%
\bibitem [{\citenamefont {Le~Gall}(2014)}]{le2014powers}%
  \BibitemOpen
  \bibfield  {author} {\bibinfo {author} {\bibfnamefont {F.}~\bibnamefont
  {Le~Gall}},\ }in\ \href@noop {} {\emph {\bibinfo {booktitle} {Proceedings of
  the 39th international symposium on symbolic and algebraic computation}}}\
  (\bibinfo {organization} {ACM},\ \bibinfo {year} {2014})\ pp.\ \bibinfo
  {pages} {296--303}\BibitemShut {NoStop}%
\bibitem [{\citenamefont {Davie}\ and\ \citenamefont
  {Stothers}(2013)}]{davie2013improved}%
  \BibitemOpen
  \bibfield  {author} {\bibinfo {author} {\bibfnamefont {A.~M.}\ \bibnamefont
  {Davie}}\ and\ \bibinfo {author} {\bibfnamefont {A.~J.}\ \bibnamefont
  {Stothers}},\ }\href@noop {} {\bibfield  {journal} {\bibinfo  {journal}
  {Proceedings of the Royal Society of Edinburgh: Section A Mathematics}\
  }\textbf {\bibinfo {volume} {143}},\ \bibinfo {pages} {351} (\bibinfo {year}
  {2013})}\BibitemShut {NoStop}%
\bibitem [{\citenamefont {Coppersmith}(1997)}]{coppersmith1997rectangular}%
  \BibitemOpen
  \bibfield  {author} {\bibinfo {author} {\bibfnamefont {D.}~\bibnamefont
  {Coppersmith}},\ }\href@noop {} {\bibfield  {journal} {\bibinfo  {journal}
  {Journal of Complexity}\ }\textbf {\bibinfo {volume} {13}},\ \bibinfo {pages}
  {42} (\bibinfo {year} {1997})}\BibitemShut {NoStop}%
\end{thebibliography}%

\clearpage
\widetext

\appendix

\begin{center}
\textbf{\large Appendix: Derivations and Motivations}
\end{center}

\section{Finding $M^n$ for a Matrix $M$ in the Normal Form}
\label{app: Generic}

\begin{theorem}
Let $M$ a matrix of the form \eqref{eq: M Definition}. Then the elements of $M^n$ are given as
\begin{equation}
\left[ M^n \right]_{i,j} = \Gamma_{i,n-(j-1)} ~~~ \forall n \in \mathbb{N}
\label{eq: General Result Appendix}
\end{equation}
where
$\left[ A \right]_{i,j}$ is the element of $A$ corresponding to $i^{th}$ row and $j^{th}$ column and $\Gamma_{i,j}$ is defined in \eqref{eq: Initialisation} and \eqref{eq: Iteration}.
\end{theorem}
\begin{proof}
Let us assume the most general form of $M^n$
\begin{equation}
M^n=
\begin{pmatrix}
\theta^{(n)}_{1,1} & \cdots & \theta^{(n)}_{1,N} \\
\vdots & \ddots & \vdots \\
\theta^{(n)}_{N,1} & \cdots & \theta^{(n)}_{N,N} \\
\end{pmatrix}
\end{equation}
for some $n \ge 1$. Then
\begin{align*}
M^{n+1} &=
\begin{pmatrix}
\theta^{(n)}_{1,1} & \theta^{(n)}_{1,2} & \cdots & \theta^{(n)}_{1,N} \\
\theta^{(n)}_{2,1} & \theta^{(n)}_{2,2} & \cdots & \theta^{(n)}_{2,N} \\
\vdots & \vdots & \cdots & \vdots \\
\theta^{(n)}_{N,1} & \theta^{(n)}_{N,2} & \cdots & \theta^{(n)}_{N,N} \\
\end{pmatrix}
\cdot
\begin{pmatrix}
\rho_1 & 1 & \cdots  & 0 \\
- \rho_2 & 0 & \cdots & 0 \\
\vdots & \vdots & \cdots & \vdots \\
(-1)^{N-1} ~ \rho_N & 0 & \cdots & 0
\end{pmatrix} \\
\begin{pmatrix}
\theta^{(n+1)}_{1,1} & \theta^{(n+1)}_{1,2} & \cdots & \theta^{(n+1)}_{1,N} \\
\theta^{(n+1)}_{2,1} & \theta^{(n+1)}_{2,2} & \cdots & \theta^{(n+1)}_{2,N} \\
\vdots & \vdots & \cdots & \vdots \\
\theta^{(n+1)}_{N,1} & \theta^{(n+1)}_{N,2} & \cdots & \theta^{(n+1)}_{N,N} \\
\end{pmatrix} &=
\begin{pmatrix}
\left(\sum_{k=1}^N (-1)^{k-1} \rho_k \theta^{(n)}_{1,k}\right) & \theta^{(n)}_{1,1} & \cdots & \theta^{(n)}_{1,N-1} \\
\left(\sum_{k=1}^N (-1)^{k-1} \rho_k \theta^{(n)}_{2,k}\right) & \theta^{(n)}_{2,1} & \cdots & \theta^{(n)}_{2,N-1} \\
\vdots & \vdots & \cdots & \vdots \\
\left(\sum_{k=1}^N (-1)^{k-1} \rho_k \theta^{(n)}_{N,k}\right) & \theta^{(n)}_{N,1} & \cdots & \theta^{(n)}_{N,N-1}
\end{pmatrix}
\end{align*}
Hence
\begin{equation}
\theta^{(n+1)}_{i,1} = \sum_{k=1}^N (-1)^{k-1} \rho_k \theta^{(n)}_{i,k} ~~~ : 1 \le i \le N, n \ge 1
\label{eq: Marker 1}
\end{equation}
and
\begin{equation}
\theta^{(n+1)}_{i,j} = \theta^{(n)}_{i,j-1} ~~~ : 1 \le i \le N,~ 1<j \le N, n \ge 1.
\label{eq: Marker 2}
\end{equation}
We now use \eqref{eq: Marker 2} iteratively to obtain
\begin{equation}
\theta^{(n+1)}_{i,j} = \theta^{(n)}_{i,j-1} = \theta^{(n-1)}_{i,j-2} = \ldots = \theta^{(n-(j-2))}_{i,1} ~~~ : 1 \le i \le N,~ 1<j \le N, n \ge 1.
\end{equation}
Now if we define 
\begin{equation}
\theta^{(n+1)}_{i,1}=\Gamma_{i,n+1} ~~~ : 1 \le i \le N
\label{eq: Marker 3}
\end{equation}
then
\begin{equation}
\theta^{(n+1)}_{i,j} = \theta^{(n-(j-2))}_{i,1} = \Gamma_{i,n-(j-2)} ~~~ : 1 \le i \le N,~ 1<j \le N, n \ge 1.
\label{eq: Marker 4}
\end{equation}
Combining \eqref{eq: Marker 3} and \eqref{eq: Marker 4}; and replacing $n+1$ by $n$ we get
\begin{equation}
\theta^{(n)}_{i,j}=\Gamma_{i,n-(j-1)} ~~~ : 1 \le i,j \le N, n \ge 2.
\label{eq: Marker 5}
\end{equation}

Now, by definition
\begin{equation}
\Gamma_{i,n+1}=\theta^{(n+1)}_{i,1} = \sum_{k=1}^N (-1)^{k-1} \rho_k \theta^{(n)}_{i,k} ~~~ : 1 \le i \le N, n \ge 1.
\end{equation}
Using \eqref{eq: Marker 5} in the right hand side gives us
\begin{equation}
\Gamma_{i,n+1}=\theta^{(n+1)}_{i,1} = \sum_{k=1}^N (-1)^{k-1} \rho_k \Gamma_{i,n-(k-1)} ~~~ : 1 \le i \le N, n \ge 1.
\end{equation}
Finally replacing $n+1$ by $j$, we have an iterative relation for $\Gamma_{i,j}$ as
\begin{equation}
\Gamma_{i,j} = \sum_{k=1}^N (-1)^{k-1} \rho_k \Gamma_{i,j-k} ~~~ : 1 \le i \le N, j \ge 2.
\end{equation}

Now note that substituting $n=1$ in \eqref{eq: Marker 5} gives
\begin{equation}
\theta^{(1)}_{i,j}=\Gamma_{i,2-j} ~~~ : 1 \le i,j \le N.
\end{equation}
However by definition of $\theta^{(n)}_{i,j}$, we have
\begin{equation}
\theta^{(1)}_{i,j} = \left[ M^1 \right]_{i,j} = M_{i,j} ~~~ : 1 \le i,j \le N.
\end{equation}
Hence,
\begin{equation}
\Gamma_{i,2-j} = M_{i,j} ~~~ : 1 \le i,j \le N
\end{equation}
which on replacing $2-j$ by $j$ yields
\begin{equation}
\Gamma_{i,j} = M_{i,2-j} ~~~ : 1 \le i \le N, -N + 2 \le j \le 1
\end{equation}
\end{proof}

\section{Motivating the Form of $a_n$}
\label{app: Motivation}

In this Appendix, we give the motivation for obtaining the fundamental sequence in the form
\begin{equation}
a_n = \sum\limits_{m=0}^{\left[\frac{n}{2}\right]} \left( -1 \right)^m ~^{n-m}C_m~ \delta^m ~ \tau^{n-2m} ~~~ \forall n \ge 0
\label{eq: a_n Definition Appendix}
\end{equation}
and $a_{-1}=0$.
In other terms, we try to understand, how the matrix
\begin{equation}
M = \begin{pmatrix}
\tau & 1 \\
-\delta & 0
\end{pmatrix}.
\label{eq: M Definition Appendix}
\end{equation}
yields
\begin{equation}
M^n=\begin{pmatrix}
a_n & a_{n-1} \\
-\delta a_{n-1} & -\delta a_{n-2}
\end{pmatrix} ~~~ \forall n \in \mathbb{N}
\label{eq: M^n Form Appendix}
\end{equation}
with $a_n$ defined in \eqref{eq: a_n Definition Appendix}.

For an matrix general $2 \times 2$ matrix $M$, if we assume
\begin{equation}
M^n = \begin{pmatrix}
a_n & b_n \\
c_n & d_n
\end{pmatrix}
\label{eq: General n power matrix}
\end{equation}
then the sequence
\begin{equation}
M, M^2, M^3, \ldots, M^n
\end{equation}
or
\begin{equation}
\begin{pmatrix}
a_1 & b_1 \\
c_1 & d_1
\end{pmatrix},
\begin{pmatrix}
a_2 & b_2 \\
c_2 & d_2
\end{pmatrix},
\begin{pmatrix}
a_3 & b_3 \\
c_3 & d_3
\end{pmatrix},
\ldots,
\begin{pmatrix}
a_n & b_n \\
c_n & d_n
\end{pmatrix}
\end{equation}
is a set of four sequences: $\left\lbrace a_n \right\rbrace, \left\lbrace b_n \right\rbrace, \left\lbrace c_n \right\rbrace$ and $\left\lbrace d_n \right\rbrace$.

The aim of the section is to show that these four sequences are restricted by constraints that allow the matrix to be expressed in terms of a single sequence.

To show this, we assume that \eqref{eq: General n power matrix} holds for the matrix defined in \eqref{eq: M Definition Appendix}. Then
\begin{align*}
M^{n+1} &= M^n.M \\
~&~\\
\begin{pmatrix}
a_{n+1} & b_{n+1} \\
c_{n+1} & d_{n+1}
\end{pmatrix}
&=
\begin{pmatrix}
a_n & b_n \\
c_n & d_n
\end{pmatrix}
\begin{pmatrix}
\tau & 1 \\
-\delta & 0
\end{pmatrix}\\
~&~\\
\begin{pmatrix}
a_{n+1} & b_{n+1} \\
c_{n+1} & d_{n+1}
\end{pmatrix}
&=
\begin{pmatrix}
\tau a_n - \delta b_n & a_n \\
\tau c_n - \delta d_n & c_n
\end{pmatrix}
\end{align*}

Hence
\begin{align*}
b_{n+1} &= a_n \\
d_{n+1} &= c_n
\end{align*}
or
\begin{equation}
b_n = a_{n-1}
\label{eq: Subs 1}
\end{equation}
\begin{equation}
d_n = c_{n-1}
\label{eq: Subs 2}
\end{equation}
and
\begin{equation}
a_{n+1} = \tau a_n - \delta b_n
\label{eq: Subs 3}
\end{equation}
\begin{equation}
c_{n+1} = \tau c_n - \delta d_n.
\label{eq: Subs 4}
\end{equation}

Substituting \eqref{eq: Subs 1} in \eqref{eq: Subs 3} and \eqref{eq: Subs 2} in \eqref{eq: Subs 4}, we get
\begin{equation}
a_{n+1} = \tau a_n - \delta a_{n-1}
\label{eq: a_n Propagation}
\end{equation}
\begin{equation}
c_{n+1} = \tau c_n - \delta c_{n-1}.
\label{eq: c_n Propagation}
\end{equation}

Hence, at the current stage
\begin{equation}
M^n =
\begin{pmatrix}
a_n & a_{n-1} \\
c_n & c_{n-1}
\end{pmatrix}
\end{equation}
with $a_n$ and $c_n$ satisfying \eqref{eq: a_n Propagation} and \eqref{eq: c_n Propagation} respectively.

Moreover, as we know
\begin{equation*}
M^1
=
\begin{pmatrix}
a_1 & a_0 \\
c_1 & c_0
\end{pmatrix}
=
M
=
\begin{pmatrix}
\tau & 1 \\
- \delta & 0
\end{pmatrix},
\end{equation*}
therefore
\begin{equation}
a_1 = \tau, ~ a_0 = 1, ~ c_1 = -\delta, ~ c_0 = 0.
\label{eq: Initial Conditions}
\end{equation}

Using \eqref{eq: a_n Propagation} and \eqref{eq: c_n Propagation} in conjugation with the initial conditions in \eqref{eq: Initial Conditions}, can write the complete sequence of $a_n$ and $c_n$. The first few terms are shown below.

\begin{minipage}{0.45\textwidth}
\begin{align*}
a_1 &= \tau \\
a_2 &= \tau^2 - \delta \\
a_3 &= \tau^3 - 2 \tau \delta \\
a_4 &= \tau^4 - 3 \tau^2 \delta + \delta^2 \\
a_5 &= \tau^5 - 4 \tau^3 \delta + 3 \tau \delta^2 \\
a_6 &= \tau^6 - 5 \tau^4 \delta + 6 \tau^2 \delta^2 - \delta^3 \\
a_7 &= \tau^7 - 6 \tau^5 \delta + 10 \tau^3 \delta^2 - 4 \tau \delta^3 \\
a_8 &= \tau^8 - 7 \tau^6 \delta + 15 \tau^4 \delta^2 - 10 \tau^2 \delta^3 + \delta^4 \\
a_9 &= \tau^9 - 8 \tau^7 \delta + 21 \tau^5 \delta^2 - 20 \tau^3 \delta^3 + 5 \tau \delta^4
\end{align*}
\end{minipage}
\hfill
\begin{minipage}{0.45\textwidth}
\begin{align*}
c_1 &= - \delta \\
c_2 &= - \delta [\tau]  \\
c_3 &= - \delta [\tau^2 - \delta] \\
c_4 &= - \delta [\tau^3 - 2 \tau \delta] \\
c_5 &= - \delta [\tau^4 - 3 \tau^2 \delta + \delta^2] \\
c_6 &= - \delta [\tau^5 - 4 \tau^3 \delta + 3 \tau \delta^2] \\
c_7 &= - \delta [\tau^6 - 5 \tau^4 \delta + 6 \tau^2 \delta^2 - \delta^3] \\
c_8 &= - \delta [\tau^7 - 6 \tau^5 \delta + 10 \tau^3 \delta^2 - 4 \tau \delta^3] \\
c_9 &= - \delta [\tau^8 - 7 \tau^6 \delta + 15 \tau^4 \delta^2 - 10 \tau^2 \delta^3 + \delta^4]
\end{align*}
\end{minipage}

It may be noted that if written in the appropriate form, it becomes clear that
\begin{equation}
c_n = - \delta a_n ~~~ \forall n \in \mathbb{N}.
\end{equation}
In order to extend the result to $c_0$, we define $a_{-1}=0$ which allows us to write $M^n$ as
\begin{equation}
M^n =
\begin{pmatrix}
a_n & a_{n-1} \\
- \delta a_{n-1} & - \delta a_{n-2}
\end{pmatrix}
\end{equation}
where
\begin{equation}
a_{n+1} = \tau a_n - \delta a_{n-1} ~~~ \forall n \ge 0
\end{equation}
and $a_{-1}=0$.

To obtain the analytic expression for $a_n$, we note the following in expansions given earlier. when terms under summation for a particular $n$ value are arranged in the ascending order of powers of $\delta$,
\begin{itemize}
\item There are $\left[ \frac{n}{2} \right] + 1$ terms in the series.
\item Powers of $\delta$ start from $0$ and increase in steps of 1.
\item Powers of $\tau$ start from $n$ and decrease in steps of 2.
\item Sign of each term alternates between plus and minus starting from a plus
\item A numerical coefficient precedes each term.
\end{itemize}
Hence $a_n$ can be written as
\begin{equation}
a_n = \sum_{m=0}^{\left[\frac{n}{2}\right]} \left( -1 \right)^m ~\eta_{m,n}~ \delta^m \tau^{n-2m}
\end{equation}

\begin{table}
\caption{List of all $\eta_{m,n}$ values}
\begin{tabular}{|c|ccccc|}
\hline
$\eta_{m,n}$ & 0 & 1 & 2 & 3 & 4 \\ \hline

1 & \color{red} ~~1~~ & ~~~~~ & ~~~~~ & ~~~~~ & ~~~~~ \\ 

2 & \color{blue} 1 & \color{red} 1 & ~ & ~ & ~ \\ 

3 & \color{magenta} 1 & \color{blue} 2 & ~ & ~ & ~ \\ 

4 & \color{cyan} 1 & \color{magenta} 3 & \color{blue} 1 & ~ & ~ \\ 

5 & \color{red} 1 & \color{cyan} 4 & \color{magenta} 3 & ~ & ~ \\ 

6 & \color{blue} 1 & \color{red} 5 & \color{cyan} 6 & \color{magenta} 1 & ~ \\ 

7 & \color{magenta} 1 & \color{blue} 6 & \color{red} 10 & \color{cyan} 4 & ~ \\ 

8 & \color{cyan} 1 & \color{magenta} 7 & \color{blue} 15 & \color{red} 10 & \color{cyan} 1 \\ 

9 & \color{red} 1 & \color{cyan} 8 & \color{magenta} 21 & \color{blue} 20 & \color{red} 5 \\ \hline

\end{tabular} 
\label{tab: eta Values}
\end{table}

A look at the $\eta_{m,n}$ values in Table \ref{tab: eta Values} reveal that
\begin{equation}
\eta_{m,n} = \eta_{m,n-1} + \eta_{m-1,n-2}
\end{equation}
which seems similar to the properties of the binomial coefficients. In fact,
\begin{equation}
\eta_{m,n} = ~^{n-m}C_m~
\end{equation}
and the structure of the Pascal's triangle might be evidently seen in the table.

\section{The Progression Rule of Fundamental Sequence}
\label{app: Fundamental Sequence}

\begin{lemma}
The $n^{th}$ term of the sequence defined in \eqref{eq: a_n Parameters} is related to its two preceding terms by the relation
\begin{equation}
a_n = \tau a_{n-1} - \delta a_{n-2}.
\label{eq: Propagation Rule Appendix}
\end{equation}
\label{le: The Lemma Appendix}
\end{lemma}

\begin{proof}
For $n=0$,
\begin{align*}
a_{n+1} &= a_1 = \tau \\
a_n &= a_0 = 1 \\
a_{n-1} &= a_{-1} = 0
\end{align*}
Hence \eqref{eq: Propagation Rule Appendix} is true for $n=0$.

For the other $n$, we prove it separately for even and odd $n$.

If $n$ is even, then
\begin{equation}
\left[\dfrac{n}{2}\right]=\dfrac{n}{2}, ~~~ \left[\dfrac{n-1}{2}\right]=\dfrac{n}{2}-1, ~~~ \left[\dfrac{n+1}{2}\right]=\dfrac{n}{2}.
\end{equation}

Now,
\begin{align*}
 \tau a_n - \delta a_{n-1} &= \tau \sum_{m=0}^{\left[\dfrac{n}{2}\right]} \left( (-1)^m ~^{n-m}C_m~ \delta^m \tau^{n-2m} \right) - \delta \sum_{m=0}^{\left[\dfrac{n-1}{2}\right]} \left( (-1)^m ~^{n-1-m}C_m~ \delta^m \tau^{n-1-2m} \right)\\
~&= \sum_{m=0}^{\dfrac{n}{2}} \left( (-1)^m ~^{n-m}C_m~ \delta^m \tau^{n+1-2m} \right) - \sum_{m=0}^{\dfrac{n}{2}-1} \left( (-1)^m ~^{n-1-m}C_m~ \delta^{m+1} \tau^{n-1-2m} \right)\\
~&= \tau^{n+1} + \sum_{m=1}^{\dfrac{n}{2}} \left( (-1)^m ~^{n-m}C_m~ \delta^m \tau^{n+1-2m} \right) - \sum_{m=0}^{\dfrac{n}{2}-1} \left( (-1)^m ~^{n-1-m}C_m~ \delta^{m+1} \tau^{n-1-2m} \right)\\
~&= \tau^{n+1} + \sum_{m=1}^{\dfrac{n}{2}} \left( (-1)^m ~^{n-m}C_m~ \delta^m \tau^{n+1-2m} \right) + \sum_{m=1}^{\dfrac{n}{2}} \left( (-1)^m ~^{n-m}C_{m-1}~ \delta^m \tau^{n+1-2m} \right)\\
~&= \tau^{n+1} + \sum_{m=1}^{\dfrac{n}{2}} \left( (-1)^m \left(~^{n-m}C_m~ + ~^{n-m}C_{m-1}~ \right) \delta^m \tau^{n+1-2m} \right)\\
~&= \tau^{n+1} + \sum_{m=1}^{\dfrac{n}{2}} \left( (-1)^m ~^{n+1-m}C_m~ \delta^m \tau^{n+1-2m} \right)\\
~&= \sum_{m=0}^{\left[\dfrac{n+1}{2}\right]} \left( (-1)^m ~^{n+1-m}C_m~ \delta^m \tau^{n+1-2m} \right)\\
~&= a_{n+1}.
\end{align*}

If $n$ is odd then
\begin{equation}
\left[\frac{n}{2}\right]=\frac{n-1}{2}, ~~~ \left[\frac{n-1}{2}\right]=\frac{n-1}{2}, ~~~ \left[\frac{n+1}{2}\right]=\frac{n+1}{2}.
\end{equation}

Now,
\begin{align*}
 \tau a_n - \delta a_{n-1} &= \tau \sum_{m=0}^{\left[\frac{n}{2}\right]} \left( (-1)^m ~^{n-m}C_m~ \delta^m \tau^{n-2m} \right) - \delta \sum_{m=0}^{\left[\frac{n-1}{2}\right]} \left( (-1)^m ~^{n-1-m}C_m~ \delta^m \tau^{n-1-2m} \right)\\
~&= \sum_{m=0}^{\frac{n-1}{2}} \left( (-1)^m ~^{n-m}C_m~ \delta^m \tau^{n+1-2m} \right) - \sum_{m=0}^{\frac{n-1}{2}-1} \left( (-1)^m ~^{n-1-m}C_m~ \delta^{m+1} \tau^{n-1-2m} \right)\\
~&= \tau^{n+1} + \sum_{m=1}^{\frac{n-1}{2}} \left( (-1)^m ~^{n-m}C_m~ \delta^m \tau^{n+1-2m} \right) - \sum_{m=0}^{\frac{n-3}{2}} \left( (-1)^m ~^{n-1-m}C_m~ \delta^{m+1} \tau^{n-1-2m} \right) - (-1)^{\frac{n-1}{2}} \delta^{\frac{n+1}{2}}\\
~&= \tau^{n+1} + \sum_{m=1}^{\frac{n-1}{2}} \left( (-1)^m ~^{n-m}C_m~ \delta^m \tau^{n+1-2m} \right) + \sum_{m=1}^{\frac{n-1}{2}} \left( (-1)^m ~^{n-m}C_{m-1}~ \delta^m \tau^{n+1-2m} \right) + (-1)^{\frac{n+1}{2}} \delta^{\frac{n+1}{2}}\\
~&= \tau^{n+1} + \sum_{m=1}^{\frac{n-1}{2}} \left( (-1)^m \left(~^{n-m}C_m~ + ~^{n-m}C_{m-1}~ \right) \delta^m \tau^{n+1-2m} \right) + (-1)^{\frac{n+1}{2}} \delta^{\frac{n+1}{2}}\\
~&= \tau^{n+1} + \sum_{m=1}^{\frac{n-1}{2}} \left( (-1)^m ~^{n+1-m}C_m~ \delta^m \tau^{n+1-2m} \right) + (-1)^{\frac{n+1}{2}} \delta^{\frac{n+1}{2}}\\
~&= \sum_{m=0}^{\left[\frac{n+1}{2}\right]} \left( (-1)^m ~^{n+1-m}C_m~ \delta^m \tau^{n+1-2m} \right)\\
~&= a_{n+1}.
\end{align*}

Hence the result is true for all $n \in \mathbb{N}$.

\end{proof}

\section{$M^n$ in Terms of Eigenvalues}
\label{app: Eigen}

\begin{theorem}
Let
\begin{equation}
M=\begin{pmatrix}
\tau & 1 \\
-\delta & 0
\end{pmatrix}
\end{equation}
with the eigenvalues
\begin{equation}
\lambda_{1,2} = \frac{\tau \pm \sqrt{\tau^2-4\delta}}{2}
\end{equation}
then
\begin{equation}
M^n=\begin{pmatrix}
a_n & a_{n-1} \\
-\delta a_{n-1} & -\delta a_{n-2}
\end{pmatrix}
~~~ \forall n \in \mathbb{N}
\end{equation}
where
\begin{equation}
a_n = \frac{\lambda_1^{n+1}-\lambda_2^{n+1}}{\lambda_1-\lambda_2}.
\end{equation}
\end{theorem}
\begin{proof}
Let $\lambda_1$ and $\lambda_2$ be the eigenvectors of $M$. Substituting the trace $\tau=\lambda_1+\lambda_2$ and determinant $\delta=\lambda_1\lambda_2$, we have
\begin{equation}
M=
\begin{pmatrix}
\lambda_1 + \lambda_2 & 1 \\
- \lambda_1 \lambda_2 & 0
\end{pmatrix}.
\end{equation}
Simple calculation of shows that the eigenvector corresponding to $\lambda_1$ is $\begin{pmatrix} 1 \\ -\lambda_2 \end{pmatrix}$ and that corresponding to $\lambda_2$ is $\begin{pmatrix} 1 \\ -\lambda_1 \end{pmatrix}$. Hence, we may construct a matrix $U$ with eigenvectors as columns,
\begin{equation}
U = \begin{pmatrix}
1 & 1 \\
-\lambda_2 & -\lambda_1
\end{pmatrix}
\end{equation}
and a diagonal matrix $D$ with the eigenvalues
\begin{equation}
D = \begin{pmatrix}
\lambda_1 & 0 \\
0 & \lambda_2
\end{pmatrix}
\end{equation}
such that
\begin{equation}
M=UDU^{-1}.
\end{equation}
and hence
\begin{equation}
M^n=UD^nU^{-1}.
\end{equation}
Substituting the values, we get
\begin{align*}
M^n &= \begin{pmatrix}
1 & 1 \\
-\lambda_2 & -\lambda_1
\end{pmatrix}
\begin{pmatrix}
\lambda_1 & 0 \\
0 & \lambda_2
\end{pmatrix}^n
\begin{pmatrix}
1 & 1 \\
-\lambda_2 & -\lambda_1
\end{pmatrix}^{-1} \\
~&~\\
~ &= \frac{1}{\lambda_2 - \lambda_1}
\begin{pmatrix}
1 & 1 \\
-\lambda_2 & -\lambda_1
\end{pmatrix}
\begin{pmatrix}
\lambda_1^n & 0 \\
0 & \lambda_2^n
\end{pmatrix}
\begin{pmatrix}
-\lambda_1 & -1 \\
\lambda_2 & 1
\end{pmatrix} \\
~&~\\
~ &= \frac{1}{\lambda_2 - \lambda_1}
\begin{pmatrix}
1 & 1 \\
-\lambda_2 & -\lambda_1
\end{pmatrix}
\begin{pmatrix}
-\lambda_1^{n+1} & -\lambda_2^n \\
\lambda_2^{n+1} & \lambda_2^n
\end{pmatrix} \\
~&~\\
~&= \frac{1}{\lambda_2 - \lambda_1}
\begin{pmatrix}
\lambda_2^{n+1} - \lambda_1^{n+1} & \lambda_2^n - \lambda_1^n \\
\lambda_1^{n+1}\lambda_2 - \lambda_2^{n+1}\lambda_1 & \lambda_1^n\lambda_2 - \lambda_2^n\lambda_1
\end{pmatrix}
\end{align*}
which can be recast as
\begin{equation}
M^n=\begin{pmatrix}
a_n & a_{n-1} \\
-\delta a_{n-1} & -\delta a_{n-2}
\end{pmatrix}
\end{equation}
with
\begin{equation}
a_n = \frac{\lambda_1^{n+1}-\lambda_2^{n+1}}{\lambda_1-\lambda_2}.
\end{equation}
\end{proof}

\section{The Form of $\phi_n$}
\label{app: phi}

\begin{corollary}
The sum of the geometric progression
\begin{equation*}
\phi_n = I + M + M^2 + \ldots + M^n
\end{equation*}
is given as
\begin{equation}
\phi_n = \begin{pmatrix}
f_n & f_{n-1} \\
- \delta f_{n-1} & 1 - \delta f_{n-2}
\end{pmatrix}
\label{eq: phi_n Form Appendix}
\end{equation}
where
\begin{equation}
f_n = \frac{1 - a_n + \delta a_{n-1}}{1 - \tau + \delta}.
\label{eq: f_n Form Appendix}
\end{equation}
\end{corollary}
\begin{proof}
Using the formula for sum of GP, we can write $\phi_n$ as
\begin{equation}
\phi_n = \frac{I - M^n}{I - M}.
\end{equation}

Hence
\begin{align*}
\phi_n &= \left[
\begin{pmatrix}
1 & 0 \\
0 & 1
\end{pmatrix}
-
\begin{pmatrix}
\tau & 1 \\
-\delta & 0
\end{pmatrix}
\right]^{-1}
\left[
\begin{pmatrix}
1 & 0 \\
0 & 1
\end{pmatrix}
-
\begin{pmatrix}
a_n & a_{n-1} \\
-\delta a_{n-1} & -\delta a_{n-2}
\end{pmatrix}
\right] \\
~ & ~ \\
~ &= \begin{pmatrix}
1-\tau & -1 \\
\delta & 1
\end{pmatrix}^{-1}
\begin{pmatrix}
1 - a_n & -a_{n-1} \\
\delta a_{n-1} & 1 + \delta a_{n-2}
\end{pmatrix} \\
~ & ~ \\
~ &= \frac{1}{1-\tau+\delta}
\begin{pmatrix}
1 & 1 \\
-\delta & 1 - \tau
\end{pmatrix}
\begin{pmatrix}
1 - a_n & -a_{n-1} \\
\delta a_{n-1} & 1 + \delta a_{n-2}
\end{pmatrix} \\
~ & ~ \\
~ &= \frac{1}{1-\tau+\delta}
\begin{pmatrix}
1 - a_n + \delta a_{n-1} & 1 - a_{n-1} + \delta a_{n-2} \\
- \delta \left( 1 - a_n \right) + \delta a_{n-1} \left( 1 - \tau \right) & \delta a_{n-1} + \left( 1 - \tau \right) \left( 1 + \delta a_{n-2} \right)
\end{pmatrix} \\
~ & ~ \\
~ &= \begin{pmatrix}
\frac{1 - a_n + \delta a_{n-1}}{1-\tau+\delta} & \frac{1 - a_{n-1} + \delta a_{n-2}}{1-\tau+\delta} \\
\frac{- \delta \left( 1 - a_n \right) + \delta a_{n-1} \left( 1 - \tau \right)}{1-\tau+\delta} & \frac{\delta a_{n-1} + \left( 1 - \tau \right) \left( 1 + \delta a_{n-2} \right)}{1-\tau+\delta}
\end{pmatrix} \\
\end{align*}

Currently, $\phi_n$ is of the form
\begin{equation}
\phi_n = \begin{pmatrix}
f_n & f_{n-1} \\
\sigma_1 & \sigma_2
\end{pmatrix}
\end{equation}
with
\begin{equation}
\sigma_1 = \frac{- \delta \left( 1 - a_n \right) + \delta a_{n-1} \left( 1 - \tau \right)}{1-\tau+\delta}
\end{equation}
and
\begin{equation}
\sigma_2 = \frac{\delta a_{n-1} + \left( 1 - \tau \right) \left( 1 + \delta a_{n-2} \right)}{1-\tau+\delta}.
\end{equation}
Using lemma \ref{le: The Lemma Appendix}, we simplify $\sigma_1$ and $\sigma_2$ as
\begin{align*}
\sigma_1 &= \frac{- \delta \left( 1 - a_n \right) + \delta a_{n-1} \left( 1 - \tau \right)}{1-\tau+\delta} \\
~ &= \frac{- \delta + \delta a_n + \delta a_{n-1} - \delta \tau a_{n-1}}{1-\tau+\delta} \\
~ &= \frac{-\delta \left( 1 - a_n - a_{n-1} + \tau a_{n-1} \right)}{1 - \tau + \delta} \\
~ &= \frac{-\delta \left( 1 - a_n - a_{n-1} + a_n +\delta a_{n-2} \right)}{1 - \tau + \delta} \\
~ &= \frac{-\delta \left( 1 - a_{n-1} +\delta a_{n-2} \right)}{1 - \tau + \delta} \\
~ &= - \delta f_{n-1}.
\end{align*}
And
\begin{align*}
\sigma_2 &= \frac{\delta a_{n-1} + \left( 1 - \tau \right) \left( 1 + \delta a_{n-2} \right)}{1-\tau+\delta} \\
~ &= \frac{\delta a_{n-1} + 1 + \delta a_{n-2} - \tau - \tau \delta a_{n-2}}{1-\tau+\delta} \\
~ &= \frac{1 - \tau + \delta \left( a_{n-1} + a_{n-2} - \tau a_{n-2} \right)}{1-\tau+\delta} \\
~ &= \frac{1 - \tau + \delta \left( a_{n-1} - \delta a_{n-3} \right)}{1-\tau+\delta} \\
~ &= \frac{1 - \tau + \delta \left( - 1 + 1 + a_{n-1} - \delta a_{n-3} \right)}{1-\tau+\delta} \\
~ &= \frac{1 - \tau + \delta + \delta \left( - 1 + a_{n-1} - \delta a_{n-3} \right)}{1-\tau+\delta} \\
~ &= \frac{1 - \tau + \delta - \delta \left( 1 - a_{n-1} + \delta a_{n-3} \right)}{1-\tau+\delta} \\
~ &= 1 - \frac{\delta \left( 1 - a_{n-1} + \delta a_{n-3} \right)}{1-\tau+\delta} \\
~ &= 1 - \delta f_{n-2}.
\end{align*}

Therefore,
\begin{equation}
\phi_n =
\begin{pmatrix}
f_n & f_{n-1} \\
-\delta f_{n-1} & 1 - \delta f_{n-2}
\end{pmatrix}.
\end{equation}

\end{proof}

\end{document}